\documentclass[11pt]{article}
\usepackage{amsmath}
\usepackage{amssymb,amsbsy,amsthm}
\usepackage{graphicx}
\usepackage[dvipsnames]{xcolor}
\usepackage{mathtools}
\usepackage{enumerate}
\usepackage{enumitem} 
\usepackage{cases}
\usepackage[margin=1in]{geometry}
\usepackage{pdfsync}
\usepackage{hyperref}
\usepackage{cleveref}
\usepackage{wrapfig} 
\usepackage{appendix}

\allowdisplaybreaks[1]
\numberwithin{equation}{section}

\usepackage{titlesec}

\DeclareMathOperator{\N}{\mathbb{N}} 

\newcommand{\p}{\mathbb{P}} 
\newcommand{\E}{\mathbb{E}} 
\DeclareMathOperator{\Cov}{Cov} 
\newcommand{\TV}{\mathrm{TV}} 

\newtheorem{theorem}{Theorem}[section]

\newtheorem{proposition}[theorem]{Proposition}

\newtheorem{fact}[theorem]{Fact}
\newtheorem{corollary}[theorem]{Corollary}
\newtheorem{conjecture}[theorem]{Conjecture}

\newtheorem{definition}[theorem]{Definition}

\newtheorem{example}[theorem]{Example}

\newtheorem{remark}[theorem]{Remark}

\title{Letter frequency in shifts of finite type with one forbidden word}

\author{
	Miklós Bóna\textsuperscript{1}
	\and
	Balázs Maga\textsuperscript{2}
	\and
	Jacob Richey\textsuperscript{2,\dag}
}

\date{\today}

\begin{document}
	
	\maketitle
	
	\begingroup
	\renewcommand{\thefootnote}{}
	\footnotetext{\textsuperscript{1} University of Florida, Gainesville, Florida, USA.}
	\footnotetext{\textsuperscript{2} HUN-REN Alfr\'ed R\'enyi Institute of Mathematics, Budapest, Hungary.}
	\footnotetext{\textsuperscript{\dag} Email: \texttt{jfrichey001@gmail.com}.}
	\endgroup


\begin{abstract} 
This work considers combinatorial and statistical aspects of {\em{shifts of finite type}}, which are families of words over a finite alphabet which avoid a fixed class of {\emph{forbidden}} sub-words. The overarching question we are interested in is: how do local statistics of a uniformly random element of the shift space depend on combinatorial features of the forbidden set? We focus on the binary alphabet $\{0,1\}$, the class of shift spaces where a single pattern is forbidden, and the average frequency of $1$s (equivalently, the probability of observing $1$ at a given position). In this case, the relevant combinatorial information is encoded by a two-variable auto-correlation polynomial associated to the forbidden word, which we call the {\em{border polynomial}}. We present several results and examples characterizing the ordering of all words by their letter frequencies: for example, we describe the set of patterns which, when forbidden, cause the frequency of $1$s to increase, decrease, or stay exactly $1/2$. Our methodologies include novel explicit local injections and bijections, generating function analysis, and a connection with a probabilistic notion of letter frequency.
\end{abstract}


\section{Introduction} \label{sec:intro}

The prototypical example of a shift of finite type (SFT), which is well-known to both symbolic dynamicists and combinatorialists, is the set of binary words with no two $1$s in a row, sometimes called the {\em{golden mean shift}} because such sequences are counted by the Fibonacci numbers. More generally, we may consider the class of families of sequences obtained by replacing the condition `no two $1$s in a row' with the condition `no contiguous subword in $\mathcal{F}$', where $\mathcal{F}$ is a fixed set of forbidden patterns, $\mathcal{F} = \{w_1,w_2,\ldots,w_r\}$, for some (finite) binary words $w_i$. Unlike pattern avoidance for permutations (see e.g.\ \cite{Vatter2015PermutationClasses}), this notion of pattern avoidance imposes a local condition, and as a result the uniform measure on the set of such sequences admits a Markovian description via the Perron-Frobenius theorem which enables exact, finite calculations. This analytic approach can be used to compute any local statistic, like the probability to observe a $1$ at a given location, or equivalently, the asymptotic frequency of $1$s in a uniformly random long word. (In particular, these quantities are always algebraic.) The details of this approach are explained for the golden mean shift in Example \ref{ex:golden}; see also Appendix \ref{app:parryconv} for the basics of this method, which is sometimes called the {\em{transfer matrix method}}, for general shifts of finite type.

However, analytic tools alone do not give much insight into how these local statistics depend on the forbidden pattern, or how the statistics of different shift spaces compare to one another. One of the motivating questions for this work was to decide the following question by combinatorial methods: in a sample sequence from the shift of finite type where $1001$ is forbidden, are $0$s or $1$s more frequent? (Jump to Section \ref{sec:freaksult} for the answer.) Generally, we have the following questions in mind: among all sets of forbidden patterns, which are the most or least costly (entropy-wise) to avoid? Which choices of forbidden patterns maximize or minimize the frequency of some fixed local pattern? And, can these questions be answered by constructing explicit embeddings or factor maps between shift spaces? Our main contributions are to give new combinatorial methods which can answer these questions, including by explicit injections and bijections.

In this work we focus on the case of binary words and shift spaces which forbid a single pattern over the base space $\mathbb{Z}$, since this is already a rich class. Many of our methods extend easily to other SFTs, and we believe that some of our constructions could be applied to higher dimensional shift spaces, a notoriously difficult case. We will stick with a discrete combinatorial viewpoint, always considering finite sequences, and often working with generating polynomials and functions. Our setup can be thought of as a statistical physics model with a hard potential, and from that perspective the frequency of 1s is typically called the {\em{magnetization}}. Questions related to optimizing local densities for a similar notion of pattern avoidance, which shares some features of our setting, was recently considered by Kenyon~\cite{Kenyon2026PatternsSequences}.

Our approach is closely related to existing work which describes the ordering of one-word SFTs by their entropy~\cite{GuibasOdlyzko1981StringOverlaps, ChandgotiaMarcusRicheyWu2026SinglePattern}. The {\em{entropy}} of a shift space, which can be thought of as the amount of randomness per symbol, is a key invariant for topological dynamical systems. The Ornstein Isomorphism theorem~\cite{Ornstein1970BernoulliShifts} states that among Bernoulli shift spaces -- random sequences generated from i.i.d. random data -- entropy is a complete invariant. The Bernoulli shifts include the measures of maximal entropy for the spaces we consider in this work (generally, any irreducible, mixing SFT), but the notion of isomorphism (measure-theoretic, between the measures of maximal entropy) is relatively weak, allowing measure zero defects.

A much stronger notion of isomorphism, which is relevant for our results, is {\em{topological conjugacy}}. For SFTs, a topological conjugacy is equivalent to a bijective sliding-block code, which can be thought of as a local, shift-commuting bijection~\cite{LindMarcus2021SymbolicDynamics}. Entropy is far from a complete invariant for topological conjugacy for SFTs in general, roughly speaking because conjugacies are rigid compared to measure-theoretic isomorphisms: they must be topological homeomorphisms everywhere, not just almost everywhere. Constructing explicit bijections between SFTs of the same entropy, or injections/surjections between SFTs of different entropy, which are in some sense locally computable -- or alternatively, disproving the existence of such maps -- remains an interesting and largely unexplored area. Theorems \ref{thm:redbeer} and \ref{thm:coffee} provide large classes of such injections and bijections (respectively) which behave nicely with respect to the frequency of $1$s. We note that although exact conditions (involving counting periodic points) for the existence of topological embeddings or factor maps are known~\cite[Theorems 10.1.1, 10.3.1]{LindMarcus2021SymbolicDynamics}, those conditions are not easy to check, and even if they are known to hold, the proofs therein do not easily allow one to exhibit explicit embeddings/factors.

By considering pattern densities from a purely statistical point of view, our work also sheds light on a new connection between properties of the measure of maximal entropy and statistics arising from a discrete combinatorial picture. Namely, Theorem \ref{thm:club} provides an explicit, quantitative connection between letter density in SFTs and a related notion of letter density, which arises from an i.i.d. sequence stopped when the forbidden word occurs (see Definition \ref{def:stop!}). While this alternate probabilistic definition measures a similar statistic in principle, and in many cases we can use the conclusion of Theorem \ref{thm:club} to prove an inequality between them, there are subtle differences between these statistics that we do not fully understand. Combining these various methods together, along with some generating function calculations, we are able to partially describe the total ordering of one-word SFTs by their letter densities, but a full description of the ordering or of the optimizers is still out of reach: see Section \ref{subsec:conj}. Of particular interest are the forbidden words which are balanced between 0s and 1s (like $1001$): among these, we prove a sufficient condition for the letter density to be exactly $1/2$ (Theorem \ref{thm:redbeer}), namely that every border of the forbidden word is itself balanced, and conjecture that it is necessary (Conjecture \ref{conj:crispy}).

\subsection{Overview} In Section \ref{sec:prelim} we state preliminary notation and definitions, collect basic facts and tools, and recall relevant previous work. In Section \ref{sec:freaksult} we include an explicit worked example for the forbidden pattern $1001$, discuss a natural, but false, monotonicity principle, and present our main theorems about the frequency of 1s in the case where a single binary word is forbidden. Section \ref{sec:proofs} contains the proofs of the main theorems. In Section \ref{sec:guessing} we present some worked examples and heuristics related to letter frequencies. Finally in Section \ref{sec:further} we summarize a number of questions left open by our results, including some ideas for future research.

\section{Preliminaries} \label{sec:prelim}

We begin by introducing notation for the classes of words and associated statistics we will study and recalling some basic facts about them. Throughout, a {\em{word}} is an element of $\{0, 1\}^k$ for some integer $k \geq 1$. The fundamental object of our study is the set of binary sequences which avoid a given word as a {\em{factor}}.

\begin{definition} \label{def:finefib} For positive integer $n \geq 1$ and any word $w$ of length $k$, denote by $\Omega_n^w$ the set of words of length $n$ with no $w$ {\bf{factor}}, i.e.
\begin{equation} \Omega_n^w = \{(\omega_1, \omega_2, \ldots, \omega_n) \in \{0, 1\}^n: \text{ for all } i \in [n-k+1], \omega_{[i, i+k-1]} \neq w \}, \end{equation}
where for any word $\omega$, $\omega_{[a,b]} = (\omega_a, \omega_{a+1}, \ldots, \omega_{b})$ if $a, b$ are valid indices for $\omega$. Denote the union over all $n$ by $\Omega^{w} = \bigcup_{n \geq 1} \Omega_n^{w}$. 

\end{definition}

By an elementary subadditivity argument, for fixed $w$, the sets $\Omega_{n}^{w}$ have an exponential growth rate: 

\begin{definition} For any word $w$, the {\bf{entropy}} $\lambda^w$ is defined as 
\begin{equation} \lambda^w = \lim_{n \to \infty} |\Omega_n^w|^{1/n} \in [1, 2). \end{equation}
\end{definition}

We are interested in counting the total number of $1$s among words in these classes. For a finite 
word $\omega$, we will use the notation $|\omega|_1 = \sum_{i \in [n]} \omega_i$ for the number of $1$s in $\omega$, and $|\omega|$ for the length of $\omega$, i.e. $|\omega| = n$ for $\omega \in \{0, 1\}^n$. We introduce notation for the level sets of $|\cdot|_1$: 

\begin{definition} For any word $w$ and positive integer $n, j \geq 1$, denote by $\Omega_{n,j}^w$ the subset of $\Omega_n^w$ of words with exactly $j$ $1$s:
\begin{equation} \Omega_{n,j}^w = \{\omega \in \Omega_{n}^w: |\omega|_1 = j\}. \end{equation}
 \end{definition}
 
The statistic we are primarily interested in is the asymptotic frequency of $1$s in a word drawn uniformly at random from $\Omega_n^w$ for large $n$.

\begin{definition} \label{def:rhos} For any word $w$ and positive integer $n$, denote by $\rho_n^w$ the frequency of $1$s over all words in $\Omega_n^w$: 

\begin{equation} \label{eq:def_rho} \rho_n^w = \frac{1}{n |\Omega_n^w|} \sum_{\omega \in \Omega_n^w} |\omega|_1 = \frac{1}{n |\Omega_n^w|}  \sum_{j=0}^n j |\Omega_{n,j}^w|. \end{equation}

Set $\rho^w = \lim_{n \to \infty} \rho_n^w \in [0,1]$ if it exists. \end{definition}

%
%
%
%
%

We note, but omit the details for brevity, that $\rho^w$ always exists and can be explicitly computed using the Parry measure ($=$ measure of maximal entropy) for the shift of finite type where $w$ is forbidden (see Appendix \ref{app:parryconv}). 

The last main character in our story arises from a probabilistic construction. 

\begin{definition} \label{def:stop!} Let $(X_n)_{n \in \N}$ be an i.i.d. sequence of Bernoulli$(1/2)$ random variables. Let $w$ be any word, and define $\tau^w \in \{1, 2, \ldots\}$ as the stopping time (w.r.t. the sequence $X$) given by  
\begin{equation} \tau^w = \min\{t \geq 1: (X_{t-k+1}, X_{t-k+2}, \ldots, X_t) = w\}. \end{equation}
For $t \in \N$ let $N(t) = \sum_{s = 1}^t X_s$ and let $Q(t) \in [0,1]$ be given by 
\begin{equation} Q(t) =\frac{N(t)}{t} = \frac{1}{t} |(X_1, \ldots, X_{t})|_1. \end{equation}
Then define the statistic $q^w = \E[Q(\tau^w)]$. 

\end{definition}

In words, $\tau^w$ records the first hitting time of the word $w$ by an i.i.d. sequence, and $Q(\tau^w)$ records the frequency of $1$s generated up to the first hitting time. Although $\E[N(\tau^w)] = \frac{1}{2}\E[\tau^w]$ by a standard martingale calculation, the ratio $Q$ and its expectation $q$ are not so easy to analyze. The tail of the random variable $\tau^w$ encodes the limiting growth rate of $\Omega_n^w$: it follows from the definitions that

\begin{equation} \p(\tau^{w} > n) = 2^{-n} |\Omega_n^w|, \end{equation}

and thus 

\begin{equation} \p(\tau^w > n) \asymp \left(\frac{\lambda^{w}}{2} \right)^n \text{ as } n \to \infty. \end{equation}

An important object is the following combinatorial polynomial, which captures the key information needed to determine $\lambda$, $\rho$ and $q$.

\begin{definition}[Borders and border polynomial] Fix two words $v, w$ with $v \in \{0,1\}^k$. Denote by $\mathcal{B}(v, w)$ the set of {\bf{borders}} of $v$ and $w$:
\begin{equation} \mathcal{B}(v,w) = \{w_{[1, j]}: j \in \{1, 2, \ldots, |w|\}, w_{[1, j]} = v_{[k-j+1, k]} \}. \end{equation}

Then define the {\bf{border polynomial}} $B^w(x,y)$ by 

\begin{equation} B^w(x, y) = \sum_{b \in \mathcal{B}(w, w)} x^{|b|} y^{|b|_1}. \end{equation}

\end{definition}

The quantities $\lambda, B$ and $\tau$ are connected by the following results. 

\begin{fact}[{\cite[Theorem 17.3.2]{LevinPeresWilmer2017MarkovChains}}] \label{fact:candy} For any word $w$, $B^w(2, 1) = \E \tau^w$. \end{fact}

This nice calculation can be done with the optional stopping theorem applied to a clever martingale stopped at time $\tau^w$. The expected value of $\tau$ also controls the entropy $\lambda$: 

\begin{theorem}[{\cite[Sec.~7]{GuibasOdlyzko1981StringOverlaps}}] \label{thm:juicy} Let $v, w$ be any words. Then $\E \tau^w \leq \E \tau^{v}$ if and only if $\lambda^w \leq \lambda^v$ if and only if $|\Omega_n^w| \leq |\Omega_n^v|$ for all $n \geq |w| \wedge |v|$. The same holds with $\leq$ replaced by $<$. 
\end{theorem}

This theorem gives a full description of how the spaces $\Omega^w$ are ordered by their entropies. One motivation of this work is to find a similar description of how the spaces $\Omega^w$ are ordered by their $\rho^w$ values, and which uses the polynomial $B$ in a simple way. One can always calculate the statistics $\lambda, \rho$ and $q$ to high precision by using explicit generating functional formulas involving $B$: 

\begin{fact} \label{fact:nerdcluster} For any word $w$, we have the exact generating function identity 
\begin{equation} \Omega^w(x,y) := \sum_{n, j \geq 0} |\Omega_{n,j}^w| x^n y^j = \left[1 - x(1+y) + B^w(x^{-1}, y^{-1})^{-1} \right]^{-1}. \end{equation}
Recall the notation from Definition \ref{def:stop!}. The joint probability generating function of $\tau$ and $N$ satisfies a similar exact formula:
\begin{equation} \Phi^w(x,y) := \E[x^{\tau^w} y^{N(\tau^w)}] = \left[1+\left(1-\frac{1}{2}x(1+y)\right)B^w(2x^{-1},y^{-1})\right]^{-1}. \end{equation}
\end{fact}

These formulas can be computed using the Goulden-Jackson cluster expansion method~\cite{GouldenJackson1979Cluster} (see also \cite{Gessel2022GouldenJackson, KupinYuster2010GJGeneralizations} for recent, similar applications of the method), and expressions for $\lambda^w$, $\rho^w$ and $q^w$ can be extracted from these generating functions by standard singularity analysis (see for example~\cite[Chapter I, Section 4]{FlajoletSedgewick2009AnalyticCombinatorics}). This gives the values of the entropy $\lambda$ and the frequency $\rho$ as polynomial and rational expressions (respectively) involving explicit degree $|w|$ polynomials, while for $q$ an integral is necessary, so transcendental functions can enter the picture. We demonstrate this for the golden mean shift; and for comparison, we do the same computations for $\lambda$ and $\rho$ using the Perron-Frobenius theorem (also known as the {\em{transfer matrix method}}).

\begin{example}[Golden mean shift] \label{ex:golden} For the forbidden word $w = 11$, we have 
\begin{equation} \Omega^{11}(x,y) = \frac{1+xy}{1-x-x^2y}, \Phi^{11}(x,y) = \frac{x^2y^2}{4-2x-x^2y}.\end{equation}
We compute $\lambda^{11} = \frac{1+\sqrt{5}}{2}$ as the inverse of the minimum positive real root of $1-x-x^2 = 0$ (the denominator of $\Omega^{11}(x, 1)$), 
\begin{equation} \rho^{11} = \lim_{n \to \infty} \frac{1}{n} \frac{[x^n] \partial_y \Omega^{11}(x,1)}{[x^n] \Omega^{11}(x,1)} = \frac{1}{\lambda^{11} (2\lambda^{11}-1)} = \frac{1}{2} - \frac{\sqrt{5}}{10} \approx .2764 , \end{equation}
and 
\begin{equation} q^{11} = \int_{0}^1 \frac{1}{x} \partial_y \Phi^{11}(x, 1) \, dx = \frac{1}{5} + \log 2 + \frac{3}{5\sqrt{5}} \log \frac{1+\sqrt{5}}{4+2\sqrt{5}} \approx .6349. \end{equation}
Alternatively, we explain how $\lambda$ and $\rho$ can be computed from the Parry measure for the golden mean shift, see Appendix \ref{app:parryconv} for more details. The minimal choice of transfer matrix is 
\begin{equation} A = \begin{pmatrix} 1&1 \\ 1 & 0 \end{pmatrix}. \end{equation} 
The matrix $A$ encodes $\Omega^{11}$ in the following sense. Let $G$ denote the graph with states $0$ and $1$ and adjacency matrix $A$ (with state $0$ corresponding to the first column/row of $A$). Then the set of all vertex paths of length $n+1$ in $G$ is in bijection with $\Omega^{11}_n$, where the bijection is: label each edge according to its terminus, and map a path to its sequence of edge labels. Note that sequences $\omega \in \Omega^{11}_{n,j}$ correspond to paths that visit vertex $1$ exactly $j$ times. 

The uniform measure on $\Omega^{11}_n$ for large $n$ induces a probability measure on $G$, called the Parry measure, which can be computed as follows. The matrix $A$ has maximum real eigenvalue $\lambda^{11} = \frac{1+\sqrt{5}}{2}$, and corresponding left/right eigenvectors $\ell = r^t = \begin{pmatrix} \lambda^{11} & 1 \end{pmatrix}$. The Parry measure $\pi$ is the probability vector 
\begin{equation} (\pi_i)_{i = 0, 1} = \left(\frac{\ell_i r_i}{\sum_j \ell_j r_j}\right)_{i = 0, 1} = \left(\frac{\lambda^{11}+1}{\lambda^{11}+2}, \frac{1}{\lambda^{11}+2}\right). \end{equation}
It follows from the (Birkhoff) ergodic theorem that $\rho^{11} = \pi_1 = \frac{1}{\lambda^{11}+2} = \frac{1}{2} - \frac{\sqrt{5}}{10}$.
\end{example}

An immediate corollary of Fact \ref{fact:nerdcluster} is that words with the same border polynomial have the same letter frequencies: 

\begin{corollary} \label{cor:lemonade} For words $w, v$ with $B^w = B^v$, and for every $n,j \in \N$ we have 
	\begin{equation} |\Omega_{n,j}^w| = |\Omega_{n,j}^v| \hspace{10pt} \text{and} \hspace{10pt} \p(\tau^w = n, N(n) = j) = \p(\tau^v = n, N(n) = j). \end{equation}
	In particular, $\lambda^w = \lambda^v$, $\rho^w = \rho^v,$ and $q^w = q^v$.  
\end{corollary}

For example, since the border polynomial $B^w$ is invariant under reversing $w$, every word has the same $\lambda$, $\rho$ and $q$ values as its reversal. In general is not clear how to give a bijective proof of the conclusion in Corollary \ref{cor:lemonade}: Theorem \ref{thm:coffee} constructs such bijections for a large class. The converse of Corollary \ref{cor:lemonade} is false: as we will see, the words $1010$ and $1100$ have frequencies $\rho$ and $q$ both equal to $1/2$. See Section \ref{sec:further} for further discussion.

\section{Results for Letter Frequencies} \label{sec:freaksult}

We now turn to the central theme of this work, which is to understand how the letter frequencies $\rho^w$ and $q^w$ are ordered over all words $w$. To get a feel for what this ordering looks like, the minimal interesting example is among words of length four with exactly two $1$s. By elementary symmetry observations (see Fact \ref{fact:tea}), the words $1100$ and $1010$ (and their reversals) have $\rho$ and $q$ values 1/2, while by an explicit computation (for example, using the generating function $\Omega$), 
$$\rho^{1001} = \frac{1}{220} \left( 11\left(5 + \sqrt{5}\right) + \sqrt{110\left(1 + 3\sqrt{5}\right)} \right) \approx 0.494161,$$

and $\rho^{0110} = 1-\rho^{1001}$ by symmetry. An exact formula is 
\begin{equation*} \rho^{1001} = \frac{\lambda^2}{-\lambda^3+5\lambda^2-\lambda-2} \text{ where } \lambda = \lambda^{1001} \text{ is the maximum real root of } \lambda^4 - 2\lambda^3 + \lambda - 1 = 0. \end{equation*}

Alternatively, we present a combinatorial proof of the fact that $\rho^{1001} < 1/2$. 

\begin{proof} \label{pf:1001} Let $A_n$ be the set of all 1001-avoiding words of length $n$ in which one $1$ is marked (we denote the mark with a hat, $\widehat{1}$ or $\widehat{0}$). Let $B_n$ be the set of all 1001-avoiding words of length $n$ in which one $0$ is marked. We will construct an injection $f:A_n \rightarrow B_n$ as follows. Let $\alpha \in A_n$, and construct $f(\alpha) \in B_n$ depending on the configuration near the $\widehat{1}$:
	\begin{itemize}
		\item ($\widehat{1} \mapsto \widehat{0}$) If we can change the $\widehat{1}$ into a $\widehat{0}$ while keeping it in place without creating a 1001-factor in $\alpha$, we make that change. Note that $f(\alpha)$ will contain a $\widehat{0}$ that is not the first entry of a 0001-factor, or the last entry of a 1000-factor.
		\item ($10\widehat{1}1 \mapsto 100\widehat{0}$) If the $\widehat{1}$ is the third entry of a 1011-factor in $\alpha$, then we cannot proceed as above. In this case we change both the $\widehat{1}$ and the $1$ after that to $0$, and mark the second one. So $10\widehat{1}1$ becomes $100\widehat{0}$.
		\item ($1\widehat{1}01 \mapsto \widehat{0}001$) Similarly, if the $\widehat{1}$ is the second entry of a $1\widehat{1}01$ factor, then change both the $\widehat{1}$ and the $1$ before it to a $0$ and mark the one on the left. So $1\widehat{1}01$ becomes $\widehat{0}001$.
	\end{itemize}
	
	We will now prove that the map $f:A_n \rightarrow B_n$ defined above is an injection. Let $\beta \in B_n$. Consider the position of the $\widehat{0}$ in $\beta$.
	
	\begin{itemize}
		\item If the $\widehat{0}$ is not the first entry of a 0001-factor or the last entry of a 1000-factor, then $f(\alpha) = \beta$ is only possible for some $\alpha$ if the first rule was used, and therefore the only preimage of $\beta$ is obtained by changing its $\widehat{0}$ to a $\widehat{1}$.
		\item If the $\widehat{0}$ is the last entry of a 1000-factor, then the second rule was used. We recover the unique preimage of $\beta$ by changing the ending $0\widehat{0}$ of that factor into $\widehat{1}1$.
		\item If the $\widehat{0}$ is the first entry of a 0001-factor, then the third rule was used. We recover the unique preimage of $\beta$ by changing the starting $\widehat{0}$ of that factor into $1\widehat{1}$.
	\end{itemize}
	
	Since $f$ is an injection, and $|A_n| = \sum_j j |\Omega_n^{1001}|$, we obtain that $\rho_n^{1001} \leq 1/2$ for all $n$. In order to prove that the inequality is sharp when $n\geq 7$ and in the limit, note that if the unique marked $\widehat{0}$ of a word $\beta \in B_n$ is the middle letter of a $100\widehat{0}001$-factor, then $\beta$ has no preimage. Indeed, it is easy to verify that in this case, trying to apply any of the three reverse rules creates a word that contains a 1001-factor. One can show that the set of words $\omega \in \Omega_n^{1001}$ containing a factor of $1000001$ is a positive fraction (uniformly over $n$) of all such words. It follows that $f$ misses a positive fraction of $B_n$ (uniformly over $n$), and thus $\rho^{1001} < 1/2$. \end{proof}

This construction is a special case of Theorem \ref{thm:redbeer}: the function $f$ defined here is exactly $\varphi_{1001}$. Also, combining this injection with Corollary \ref{cor:kave}, we find that $q^{1001} > 1/2$. 

\subsection{A Monotonicity Principle}\label{subsec:conj} 
We focus on the ordering of frequency values among words of the same length, and begin by mentioning a natural monotonicity principle describing this ordering, which we wrote as a conjecture in the original version of this paper, but which turned out to be false. The monotonicity principle is the rule of thumb that the more $1$s in a word $w$, the smaller the $\rho^w$ value should be, and the larger the $q^w$ value should be. In other words, one might guess that $\rho^w$ is monotone decreasing with respect to $|w|_1$, and $q^w$ is monotone increasing with respect to $|w|_1$. Here is a simple heuristic justifying this guess. We can think of forming the sets $\Omega^v$ and $\Omega^w$ by starting with all binary words, and removing those which have at least one copy of $v$ or $w$, respectively. If $v$ has fewer $1$s than $w$, then it is intuitive that fewer $1$s are removed to form $\Omega^v$ than to form $\Omega^w$, thus leaving a higher density of $1$s in $\Omega^v$ than in $\Omega^w$. We expect the ordering by $q$ values to be exactly the opposite because $q$ measures the frequency of $1$s over words that have exactly one copy of the forbidden word (at the end), so in a sense $q^w$ lives in the complement of $\Omega^w$. 

This heuristic, while somewhat appealing, hides an important subtlety: the statistic $\rho$ is normalized by the size of $\Omega$, which depends on the forbidden word through the entropy $\lambda$, which is not at all monotone in $|w|_1$. Write $\mathcal{R}_k$ for the set of four words $1^{k-1}0, 10^{k-1}, 0^{k-1}1, 01^{k-1}$, and $\mathcal{R} = \cup_{k \geq 2} \mathcal{R}_k$. The words in $\mathcal{R}$ and the constant words are easily seen to violate the monotonicity principle: by direct calculations, for $k \geq 6$, 

\begin{equation}
	\rho^{1^{k-1}0} < \rho^{1^k}, \qquad q^{1^{k-1}0} > q^{1^k}, \qquad \rho^{0^k} < \rho^{101^{k-2}}, \qquad q^{0^k} > q^{010^{k-2}}
\end{equation}

Morally, these six words cause trouble because the words in $\mathcal{R}_k$ are essentially duplicates of the constant words $1^{k-1}$ and $0^{k-1}$, which have length $k-1$ instead of $k$. For example, forbidding $01^{k-1}$ is essentially equivalent to forbidding $1^{k-1}$ in the sense that their generating function asymptotics and in particular all their statistics $\lambda$ and $\rho$, are equal. (An equivalent definition of $\mathcal{R}$ is those words which, when forbidden, do not give an irreducible shift of finite type.) 

Even excluding the constant words and words in $\mathcal{R}$, the principle appears to be false in general, although one has to look at large words to see the first counterexample. Working with ChatGPT 5.6 a few months after the original version of this paper was uploaded to the arXiv, we found a class of candidate pairs (of length $k \geq 41$) which appear to violate the monotonicity principle, i.e. pairs of words $v,w$ such that $|v|_1  < |w_1|$, $\rho^v < \rho^w$ and $q^v > q^w$. We plan to present proofs and a general quantitative description of these counterexamples, which appear to be rare, in forthcoming work.

Theorem \ref{thm:coffee} and Proposition \ref{prop:selfless} establish the monotonicity principle for large families of pairs $(v,w)$ with the same entropy $\lambda^w = \lambda^v$, while Theorem \ref{thm:club} relates the ordering by $q$ values to the ordering by $\rho$ values. Describing the full orderings would also require understanding the situation among words with the same length {\em{and}} the same number of $1$s, where different values can occur. At this level of detail, the comparison between $\rho$ and $q$ breaks down (an example is the pair $v = 10001$, $w = 01010$, where both $\rho^v > \rho^w$ and $q^v > q^w$). We are still lacking a simple conjectural description of the ordering in these cases. See Section \ref{sec:further} for further discussion.


\subsection{Main Results}

We now present our main results, which give proofs of inequalities between letter frequencies of different words, in many cases by constructing explicit, novel bijections. In particular, we establish that the monotonicity principle holds in the case where there is nice cross-border structure between the two words, and among the class of words which have no non-trivial borders. We also describe a general method which allows us to directly compare the frequencies of $0$s and $1$s, and thus to test whether $\rho^w$ is less than or equal to $1/2$. We use this method to partially classify words with $\rho = 1/2$, and prove $\rho < 1/2$ for a large class.

As a warmup, we point out one easy symmetry present in our setup: 

\begin{fact} \label{fact:tea} Let $w$ be any word whose reversal is equal to its bitflip. Then $\rho^w = q^w = 1/2$. \end{fact}

There is a simple bijective proof for $\rho$, but for $q$ we only know a generating function proof, which we defer until Proposition \ref{prop:baboosh}. 

\begin{proof} The map Reverse $\circ$ Bitflip is an involution on $\Omega_n^w$ which restricts to a bijection $\Omega_{n,j}^w \to \Omega_{n, n-j}^w$ for all $n, j$, since if $\omega$ avoids $w$, then the bitflip-reversal of $\omega$ avoids the bitflip-reversal of $w$, which is equal to $w$. \end{proof} 

As a small example, the words $1010$ and $1100$ are the only words of length $4$ that have this property. While bitflip-reversal symmetry is nice and gives a simple bijection, it applies in a very small proportion of cases: among words of even length $2k$, there are exactly $2^k$ many such words, an asymptotically vanishing fraction of the total. Consider the following (strictly weaker) property: 

\begin{definition} A word $w$ has {\bf{balanced borders}} if for all $b \in \mathcal{B}(w,w)$, $|b| = 2|b|_1$. \end{definition}

Any word that is its own bitflip reversal is easily seen to have balanced borders. An example of a word with balanced borders, but that is not its own bitflip-reversal, is $10011010$. Note that, for example, $11000011$ does not have balanced borders, even though the word itself is balanced. The balanced borders property is equivalent to the following strong symmetry property: 

\begin{proposition} \label{prop:baboosh} A word $w$ has balanced borders if and only if for all $n, j \in \N$, 
	\begin{equation} |\Omega^w_{n,j}| = |\Omega^w_{n,n-j}| \text{ and } \p(\tau^w = n, N^w = j) = \p(\tau^w = n, N^w = n-j).
	\end{equation}
	In particular, if $w$ has balanced borders, then $\rho^w = q^w = \frac{1}{2}$. 
\end{proposition}

\begin{proof}[Proof of Proposition \ref{prop:baboosh}] We use the explicit formulas for the generating functions in Fact \ref{fact:nerdcluster}. Observe that any monomial of the form $x^{2k} y^k$ for $k \in \N$ is invariant under the transform $(x,y) \mapsto (xy, y^{-1})$. The term $x(1+y)$ is also invariant under this transform. It follows that $B^w$, and thus also $\Omega^w$ and $\Phi^w$, are invariant under this transform if $w$ has balanced borders. The same transform implies the stated symmetry of the coefficients. Since the series $\Omega$ and $\Phi$ both have positive radius convergence as two-variable real analytic functions -- specifically, $\Omega^w$ converges for $x < (\lambda^w)^{-1}$ and $y \leq 1$, and $\Phi^w$ converges for $x < 2(\lambda^w)^{-1}$ and $y \leq 1$ -- uniqueness of analytic functions gives the reverse implication. \end{proof}

We now state our two main theorems, which give explicit injections (and sometimes bijections) between different spaces $\Omega$. Here we do the minimum amount of setup necessary to describe the results, and postpone the details of constructing the relevant maps $\psi$ and $\varphi$ to Section \ref{sec:clusters}. We will need a marked version of the spaces $\Omega$ (akin to marking a single $0$ or $1$):

\begin{definition} \label{def:mark1} For a word $w$ and $\epsilon \in \{0,1\}$, let $\Omega_n^w(\epsilon)$ denote the set of sequences in $\Omega_n^w$ where a single index is marked:
\begin{equation} \Omega_n^w(\epsilon) = \{(\omega, j) \in \Omega_n^w \times [n]: \omega_j = \epsilon\}. \end{equation}
\end{definition} 

\begin{theorem}\label{thm:redbeer} Let $w$ be any word such that for all $b \in \mathcal{B}(w,w)$, $2|b|_1 \geq |b|$. Then there exists an explicit length-preserving injection $\varphi_w: \Omega^w(1) \to \Omega^w(0)$ with coding radius $|w|$. If strict inequality $2|b|_1 > |b|$ holds for some $b \in \mathcal{B}(w,w)$, then for every $n \in \N$, there exists a subset $A = A(n,w) \subset \Omega_n^w(0)$ and a constant $c = c(w) > 0$ such that
\begin{equation} |A| \geq c |\Omega_n^w(0)| \text{ and } A \cap \text{Range}(\varphi_w) = \emptyset. \end{equation}
In particular, $\rho^w \leq \frac{1}{2}$. If strict inequality holds in the assumption for some border $b$, then $\rho^w < \frac{1}{2}$. 
\end{theorem}

Here the {\em{coding radius}} being $|w|$ means that for any $\omega \in \Omega_n^w(1)$ and $i \in [n]$, $\varphi_w(\omega)_i$ is a function of only the position of the mark and the
$2|w|$ digits $\omega_{i-|w|+1}, \omega_{i-|w|+2}, \ldots, \omega_{i + |w|}$. 
In fact, it is a fixed function (not depending on $i$) of those digits and the relative position of the mark in this interval. (In the language of symbolic dynamics, $\varphi$ is like a sliding-block code.)


Applying Theorem \ref{thm:redbeer}, and again the same theorem with $0$s and $1$s exchanged, gives a bijective proof of Proposition \ref{prop:baboosh}. Computer computations suggest that the condition in Theorem \ref{thm:redbeer} can be applied to around $85\%$ of words of (large) even length and $80\%$ of words of (large) odd length, in the sense that either the condition $2|b|_1 \leq |b|$ holds for all $b$ or the same condition holds for the bitflip of $w$ as $|w| \to \infty$. Among the balanced words $w$, i.e. satisfying $2|w|_1 = |w|$, of large (necessarily even) length, this asymptotic proportion is around $99.6\%$. 


On the other hand, there are interesting examples where Theorem \ref{thm:redbeer} cannot be applied: the minimal example is $10001$, which itself is $0$-heavy, but has the border $1$ which is $1$-heavy. An example of a balanced word where the theorem cannot be applied is $100011110001$, which has correlation polynomial $B^v(x,y) = x^{12}y^6 + x^5y^2 + xy$, with one $1$-heavy border (of length 1) and one $0$-heavy border (of length $5$). (Still, in this case, the proof of Theorem \ref{thm:redbeer} -- in particular, the construction of the map $\varphi$ -- suggests a natural guess for whether $\rho > $ or $< 1/2$: see Remark \ref{rem:guess}.)

Next, we present a result which establishes the monotonicity principle under an assumption on the borders of $v, w$.

\begin{theorem} \label{thm:coffee} Let $v, w$ be any two different words with $|v| = |w|$ and $|v|_1 \geq |w|_1$. Assume that
\begin{equation} \mathcal{B}(v, v) \setminus \{v\} = \mathcal{B}(w,w) \setminus \{w\}. \end{equation}
Then there exists an explicit, length-preserving bijection $\kappa_{v,w}: \Omega^w \to \Omega^v$ such that $|\omega|_1 \geq |\kappa_{v,w}(\omega)|_1$ for all $\omega \in \Omega^w$. 

If in addition 
\begin{equation} \mathcal{B}(v, v) \setminus \{v\} = \mathcal{B}(w,w) \setminus \{w\} = \mathcal{B}(v,w) =  \mathcal{B}(w,v), \end{equation}
then there exists a (different) bijection $\psi_{v,w}: \Omega^w \to \Omega^v$ with all the same properties as $\kappa$, plus $\psi_{v,w}$ has bounded coding radius.

In either case, $\rho^v \leq \rho^w$. Moreover, if $|v|_1 > |w|_1$, then $\rho^v < \rho^w$. 
\end{theorem}

The maps $\kappa$ were shown to be bijections in \cite[Theorem 3.1]{CarriganHollarsRowland2024NaturalBijection}, while the maps $\psi$ were shown to be topological conjugacies in \cite[Proposition 4.11]{ChandgotiaMarcusRicheyWu2026SinglePattern}. Both maps use the same basic idea -- namely, to replace copies of $v$ in a given $\omega \in \Omega^w$ with copies of $w$ -- but in different ways: in $\kappa$, occurrences of $v$ are replaced with $w$ iteratively, by searching left-to-right, while in $\psi$ clusters of $v$ (cf. Definition \ref{def:cluster}) are replaced all at once with clusters of $w$.

The border conditions appearing in Theorem \ref{thm:coffee} are easy to check and cover a large class of pairs. An example where the condition for $\psi$ holds is the pair $v=101001, w=110001$, and an example where the condition for $\kappa$ holds but the one for $\psi$ does not is $v = 100111, w = 111001$. Computer computations suggest that the condition for $\kappa$ holds for around $13\%$ of all pairs asymptotically, while the stronger condition for the existence of $\psi$ holds for around $5\%$ of pairs. (It can be proved that the stronger condition applies to a uniformly positive proportion of all pairs $v, w$ of arbitrarily large length by considering words which have no non-trivial borders or cross borders. One way to do this is to consider words that begin with $00$ and end with $11$, which guarantees that no borders or cross-borders of length $\leq 3$ occur; then the larger borders can be ruled out by a union bound.)


We believe that the maps $\varphi, \kappa$ and $\psi$ used to obtain these theorems could, with a bit more effort, be applied to obtain similar comparison results for other pattern densities in more general SFTs, including higher dimensional SFTs. See Section \ref{sec:further} for further discussion.

Finally, we note that for a large class of pairs, which do not necessarily satisfy the assumption of Theorem \ref{thm:coffee}, the monotonicity principle can be checked by an explicit generating function calculation. These are the words which have no non-trivial borders. 

\begin{proposition} \label{prop:selfless} Let $v, w$ be any words with $|v| = |w| \geq 3$, $\mathcal{B}(v,v) = \{v\}$, $\mathcal{B}(w,w) = \{w\}$, and $|v|_1 \leq |w|_1$. Then $\rho^w \leq \rho^v$ and $q^v \leq q^w$. The same holds if $\leq$ is replaced by $<$. \end{proposition}

\begin{proof} Consider words of length $k$ with no non-trivial borders, i.e. with border polynomial $B(x,y) = x^k y^j$ for some integer $0 \leq j \leq k$. For any word $w$ with $|w| = k$ and $|w|_1 = j$, $\lambda = \lambda^w$ does not depend on $j$ and satisfies $\lambda^{k} - 2\lambda^{k-1} + 1 = 0$. Standard singularity analysis using the generating function $\Omega^w(x,y)$ gives 
	\begin{equation} \rho^w = \frac{\lambda^{k-1} - j}{2\lambda^{k-1} - k}. \end{equation}
	
	Thus, for fixed $k$, $\rho^w$ is a strictly decreasing (linear) function of $j = |w|_1$. An exact calculation with the generating function $\Phi^w$ gives 
	\begin{equation} 
		q^w = \frac{1}{2} + \left(j - \frac{k}{2}\right) \int_{0}^\infty \frac{t(t+2)^{k-2}}{(1+t(t+2)^{k-1})^2} \, dt .
	\end{equation}
	
	The integral is easily seen to be positive and finite for $k \geq 3$, confirming that $q$ is a strictly increasing (again, linear) function of $j = |w|_1$. 
\end{proof}

The trivial border polynomials appearing here actually represent the most common case: around $26.78\%$ of all words have no non-trivial borders~\cite{GuibasOdlyzko1981Periods} (see also \url{https://oeis.org/A003000}, \url{https://oeis.org/A242430}). Many such pairs satisfy the assumption of Theorem \ref{thm:coffee}, giving us a bijective proof, but not all do: for example, any word $w$ with no non-trivial borders and its reversal will always have a cross-border of length $1$ (e.g. the pair $1100$ and $0011$, where $\mathcal{B}(1100, 0011) = \{0, 00\}$).



\subsection{Relating $\rho$ and $q$} 

We now turn to the letter frequency $q^w$ obtained from drawing i.i.d. letters until the word $w$ occurs, with the aim of showing that the $q^w$ values are ordered exactly opposite to the $\rho^w$ values. To justify this heuristic, we prove the following explicit expression relating $\rho$ and $q$: 

\begin{theorem} \label{thm:club} For any word $w$, 
	\begin{equation} q^w = \frac{1}{2}+\sum_{n=|w|}^{\infty}\frac{\mathbb{P}(\tau^w>n)}{n+1}\left(\frac{1}{2}-\rho_n^w\right), \end{equation}
	where $\p$ denotes i.i.d. Bernoulli$(1/2)$ measure. 
\end{theorem}

Theorem \ref{thm:club} quickly gives the following two corollaries relating the $\rho$ and $q$ orderings.  

\begin{corollary} \label{cor:kave} If $\rho_n^w \leq \frac{1}{2}$ for all $n$ and $\rho^w < \frac{1}{2}$, then $q^w > \frac{1}{2}$. \end{corollary} 

\begin{proof} The inequality $\rho_n^w \leq \frac{1}{2}$ guarantees that all terms in the sum in Theorem \ref{thm:club} are non-negative, and the strict inequality $\rho^w < \frac{1}{2}$ guarantees that all terms are eventually strictly positive. \end{proof}

\begin{corollary} \label{cor:haz} Let $v, w$ be such that $\lambda^v \leq \lambda^w$ and $\rho_n^w \leq \rho_n^v \leq \frac{1}{2}$ for all $n$. Then $q^v \leq q^w$. 
\end{corollary}

\begin{proof} By Theorem \ref{thm:juicy}, $\lambda^v \leq \lambda^w$ is equivalent to $|\Omega^v_n| \leq |\Omega^w_n|$ for all $n$. Since $\p(\tau^w > n) = 2^{-n}|\Omega_n^w|$, the result follows by plugging into the sum in Theorem \ref{thm:club}. \end{proof}

Corollary \ref{cor:kave} says that $\rho$ and $q$ give the opposite ordering when comparing the frequency to the unbiased value $\frac{1}{2}$, while Corollary \ref{cor:haz} relates the two orderings more generally, under an assumption on the entropy. Comparing the entropies is easy by Theorem \ref{thm:juicy}, but comparing $\rho_n^v$ to $\rho_n^w$ boils down to proving an inequality between two linear recursions, a notoriously difficult task. Computer simulations suggest that the assumption $\rho_n^v \leq \rho_n^w$ is weaker than the assumption $\rho^v \leq \rho^w$ (excluding the reducible words $\mathcal{R}$), so it seems likely the first assumption in Corollary \ref{cor:kave} is not necessary. See Section \ref{sec:further} for further discussion.

\section{Proofs} \label{sec:proofs}

Here we prove the main results of Section \ref{sec:freaksult}. 

\subsection{Cluster maps} \label{sec:clusters}

In this section we construct the maps $\varphi$ and $\psi$ appearing in the theorems of Section \ref{sec:freaksult}. To start, we introduce the notion of a cluster, which will be central to both constructions. 

\begin{definition} \label{def:cluster} Fix a word $v \in \{0,1\}^k$. Call a word $c$ a {\bf{cluster}} of $v$ if every digit of $c$ is part of a factor of $v$ in $c$, that is, if $|c| = n$, and for all $i \in [n], c_{[i-j+1, i-j+k]} = v$ with some $1\leq j\leq k$. Denote the set of all clusters of $v$ by $\mathcal{C}^v$. \end{definition}

The clusters are all possible ways of packing copies of $w$ together with overlaps. Note that for any fixed word $v$, any word $\omega$ can be uniquely decomposed into $v$ clusters and $v$ non-clusters: 

\begin{fact} \label{fact:clustercomp} For any words $v$ and $\omega$, there is a unique decomposition of $\omega$ into clusters of $v$: namely, there exist words $a_1, a_2, \ldots, a_{r+1}$ and $c_1, c_2, \ldots, c_{r+1} \in C^v$ such that 
\begin{equation} \label{eq:decomp} \omega = a_1 c_1 a_2 c_2 \cdots a_r c_r a_{r+1}, \end{equation}
and additionally: for all $i$, $v_{[2, k]}  a_i  v_{[1,k-1]} \in \Omega^v$; and none of the $a_i$ are the empty word except possibly $a_1$ and/or $a_{r+1}$. 
\end{fact}

\subsection{Proof of Theorem \ref{thm:coffee}}

\begin{proof}[Proof of Theorem \ref{thm:coffee}] \cite[Theorem 3.1]{CarriganHollarsRowland2024NaturalBijection} gives the desired bijection: namely, the map $\phi_L$, with words $p = v$ and $q = w$, which repeatedly replaces the left-most $v$ factor in $\omega \in \Omega^w$ with $w$, is the desired bijection $\kappa_{v,w}$. Since each replacement of $v$ by $w$ weakly decreases the number of $1$s, the same property holds for $\kappa$. Thus $\rho^v \leq \rho^w$. We defer the strict inequality case until the end, where we will deal with it simultaneously for both $\kappa$ and $\psi$. 
	
We now turn to the construction of the map $\psi_{v,w}$. We start by showing that the assumption $\mathcal{B}(v,v) \setminus \{v\} = \mathcal{B}(w,w) \setminus \{w\}$ alone implies that the clusters $\mathcal{C}^v$ and $\mathcal{C}^w$ are in (length-preserving) bijection. Define a map $g_{v,w}: C^v \to C^w$ as follows: for $c \in C^v$ and any index $j \in [|c|]$, $c_j$ belongs to at least one factor of $v$ in $c$. Pick any such factor, and write $c_j = v_i$, so $c_j$ is the $i$th digit in that factor of $w$, for some $i \in [|w|]$. Then we define the image map at location $j$ by $g_{v,w}(c)_j = w_i$. The map does not depend on which factor of $v$ was chosen since they all agree at location $j$ in $c$, so $g$ is well-defined. Reversing the roles of $v$ and $w$ shows that $g$ is an involution, and hence a bijection. 

We now upgrade the map $g_{v,w}$ to an involution $\psi_{v,w}: \Omega^w \to \Omega^v$ using Fact \ref{fact:clustercomp}. Given $\omega \in \Omega^w$ with unique $v$ cluster decomposition as in Equation \eqref{eq:decomp}, define
\begin{equation} \label{eq:imagesum} \psi_{v,w}(\omega) = a_1 g_{v,w}(c_1) a_2 \cdots g_{v,w}(c_r) a_{r+1}. \end{equation}
Since the decomposition is unique, the map is well-defined, but it remains to check that the above expression is the unique $w$ cluster decomposition, and that the image $\psi_{v,w}(\omega)$ belongs to $\Omega^v$. Note that since $g$ preserves the cluster length, $\psi$ preserves word length. 

Fix $\omega \in \Omega^w$ with $v$ cluster decomposition as in Fact \ref{fact:clustercomp}. First observe that by the definition of the cluster decomposition, no digit of any non-cluster factor $a_i$ may belong to a $v$ factor in $\omega$. The same holds for $w$ factors since $\omega \in \Omega^w$. 

Suppose for the sake of contradiction that some digit of a non-cluster part $a_i$ belongs to a $w$ factor in $\psi_{v,w}(\omega)$. Since $\omega \in \Omega^w$, no $a_i$ contains any $w$ factor, so assume that some digit of $a_i$ belongs to a $w$ factor in the word $a_i  g_{v,w}(c_i)$. (The other possibility is that some digit of $a_i$ belongs to a $w$ factor in the word $g_{v,w}(c_{i-1}) a_i$, but the same proof works for both cases.) It follows that $\mathcal{B}(a, w)$ contains a prefix $p$ of $w$ such that $p b = w$ for some $b \in \mathcal{B}(w,w)$. By the assumption $\mathcal{B}(w,v) = \mathcal{B}(w,w) \setminus \{w\}$, $c_i$ has the word $b$ as a prefix, and thus $a_i b$ is a factor of $\omega$, a contradiction since the latter contains $w$ as a factor. 

Similarly, suppose for the sake of contradiction that some non-cluster part $a_i$ belongs to a $w$ factor in $\psi_{v,w}(\omega)$ which also shares digits in common with the neighboring clusters $g_{v,w}(c_{i-1})$ and $g_{v,w}(c_i)$. Write that $w$ factor as $w = b  a_i  b'$ for some borders $b, b' \in \mathcal{B}(w,w) \setminus \{w\}$. Then since $\mathcal{B}(w,w) \setminus \{w\} = \mathcal{B}(v,w) = \mathcal{B}(w,v)$, $b$ is also a suffix of $v$, and similarly $b'$ is a prefix of $v$. By definition of the map $g$, $c_{i-1}$ has $b$ as a suffix, and $c_i$ has $b'$ as a prefix. Thus $b  a_i  b' = w$ is a factor of $\omega$ which connects the two clusters $c_{i-1}$ and $c_i$, contradicting the assumption that $\omega \in \Omega^w$. Combining the above, we have shown that the formula \eqref{eq:imagesum} is the unique $w$ cluster decomposition of the image word $\psi_{v,w}(\omega)$. 

It remains to check that the image avoids $v$ as a factor. We noted previously that no $a_i$ may contain $v$ as a factor. Suppose for the sake of contradiction that one of the cluster parts $g_{v,w}(c_i)$ contains $v$ as a factor. Then since $|v| = |w|$ and $g_{v,w}(c_i) \in \mathcal{C}^w$, that $v$ factor is covered by two overlapping $w$ factors in $g_{v,w}(c_i)$: that is, there exist borders $b, b' \in \mathcal{B}(w,v) = \mathcal{B}(v,w)$ such that $v = bs = tb'$ and $|b| + |b'| > |v|$. Since the cross border sets are equal to $\mathcal{B}(w,w) \setminus \{w\} = \mathcal{B}(v,v) \setminus \{v\}$, $b$ is simultaneously a prefix of both $v$ and $w$, and $b'$ is simultaneously a suffix of both $v$ and $w$. Since $|b| + |b'| > |v|$, we reach the contradiction $v = w$. 

Similarly, suppose for the sake of contradiction that some non-cluster part $a_i$ belongs to a $v$ factor in $\psi_{v,w}(\omega)$ which also shares digits in common with the neighboring clusters $g_{v,w}(c_{i-1})$ and $g_{v,w}(c_i)$. Write that $v$ factor as $v = b  a_i  b'$ for some borders $b \in \mathcal{B}(w, v), b' \in \mathcal{B}(v, w)$. Then since $\mathcal{B}(w,v) = \mathcal{B}(v,v) \setminus \{v\}$, $b$ is also a suffix of $v$, and similarly $b'$ is a prefix of $v$. By definition of the map $g$, $c_{i-1}$ has $b$ as a suffix, and $c_i$ has $b'$ as a prefix. Thus $b  a_i  b' = v$ is a factor of $\omega$ which connects the two clusters $c_{i-1}$ and $c_i$, contradicting the assumption that we started with the unique cluster decomposition. Combining the above, we have shown that each factor $g_{v,w}(c_i)$ is a maximal cluster, and hence the formula \eqref{eq:imagesum} is the unique $w$ cluster decomposition of the image word. This completes the proof that $\psi_{v,w}$ is the desired involution.

Finally, note that for any cluster $c \in C^v$, any digit of $c$ that belongs to multiple $v$ factors is left unchanged by the map $g_{v,w}$, since any such digit belongs to a border of $v$ which exactly agrees with the corresponding border of $w$. Thus the only digits of $c$ which are possibly changed by $g$ are those at the edges of the cluster $c$: namely, the map $g_{v,w}$ transforms the border complements $v_{[1,k-j]}$ or $v_{[k-j+1, k]}$ when $|v| = k$ and $v_{[1,j]} \in \mathcal{B}(v,v)$ to the corresponding border complements for $w$, i.e. $w_{[1, k-j]}$ or $w_{[k-j+1, k]}$. For any such border $b = v_{[1, j]}$ with corresponding border complements $s = v_{[j+1, k]}$ and $t = w_{[j+1, k]}$, 
\begin{equation} |v|_1 \geq |w|_1 \implies |s|_1 = |v|_1 - |b|_1 \geq |w|_1 - |b|_1 = |t|_1. \end{equation}
It follows that for any $\omega \in \Omega^w$, $|\omega|_1 \geq |\psi_{v,w}(\omega)|_1$. 

Finally, suppose $|v|_1 > |w|_1$. Then every occurrence of the cluster $c_i = v$ in $\omega$ is transformed to $g_{v,w}(v) = w$, decreasing the number of $1$'s by at least one, and the same holds for the successive replacement map $\kappa$. By Corollary \ref{cor:parrypals}, a uniformly random element of $\omega \in \Omega^w_n$ has $\Theta(n)$ many such clusters with high probability, implying $\rho^v < \rho^w$. 
\end{proof}

\subsection{Proof of Theorem \ref{thm:redbeer}}

To construct the maps $\varphi_w$, we will also build basic maps related to clusters which can then be extended using the decomposition of Fact \ref{fact:clustercomp}. There is an essential difference to keep in mind: while the maps $\psi_{v,w}$ mapped clusters of $v$ to clusters of $w$, the maps $\varphi_w$ will map near-clusters of $w$ to near-clusters of $w$, where `near-cluster' means a cluster with a single digit bitflipped. To start, we will need a more restrictive notion of cluster, which zooms in on the central overlap region. 

\begin{definition} \label{def:center} Fix a word $w$, and let $c \in \mathcal{C}^w$ be a cluster. Define $I(c)$, the {\bf{common intersection}} of $c$, as the factor of $c$ consisting of all digits which belong to every factor of $w$ in $c$. We call $c$ a {\bf{compact cluster}} if $I(c) \neq \emptyset$. Denote the set of compact clusters by $\mathfrak{C}^w \subset \mathcal{C}^w$. 
\end{definition}

By observing the first and last digits of a cluster, we see that any cluster starts and ends with a factor of $w$, thus 
a cluster being compact simply means that these two overlap, implying that any $c\in \mathfrak{C}^w$ has length at most $2|w|-1$. 
This overlapping part is an element of $\mathcal{B}(w, w)$ 
and it coincides with the common intersection of $c$. Thus there is a natural bijection between $\mathfrak{C}^w$ and $\mathcal{B}(w, w)$.
Consequently, $1\leq |\mathfrak{C}^w|\leq |w|$. ($w \in \mathfrak{C}^w$ always holds.) In contrast, when $\mathcal{B}(w,w) \neq \{w\}$, $\mathcal{C}^w$ is always infinite, containing all words consisting of consecutive overlapping copies of $w$. Observe in particular the equivalence 
\begin{equation} \forall b \in \mathcal{B}(w,w),  2|b|_1 \geq |b|  \iff \forall c \in \mathfrak{C}^w, 2|I(c)|_1 \geq |I(c)|. \end{equation}
Note this also holds if $\geq$ is replaced by $=$, which is exactly the balanced borders condition. 

Recall the sets of marked words $\Omega^w(\epsilon)$. The following notation will be useful: 
\begin{definition} For any marked sequence $\alpha = (\omega, j)$ with $j \in [|\omega|]$ and $\epsilon \in \{0,1\}$, denote by $\alpha^\epsilon$ the marked word obtained by setting the marked letter in $\alpha$ to be $\epsilon$, i.e. $\alpha^{\epsilon} = (\omega', j)$ where 
	\begin{equation} \omega'_j = \begin{cases} \epsilon, & i = j \\ \omega'_i = \omega_i, & i \neq j \end{cases}\end{equation} 
\end{definition}

Next we define a class of marked words which are one digit away from being compact clusters. These words will form the basis of the maps $\varphi_w$. 

\begin{definition} Fix a word w. For $\epsilon \in \{0,1\}$, denote by $\mathfrak{C}^w(\epsilon)$ the set of marked words $(c^{\epsilon},j) \in \Omega^w(\epsilon)$ obtained by marking a compact cluster $c \in \mathfrak{C}^w$ at some digit, necessarily with value $1-\epsilon$, at index $j \in [|I(c)|]$.   \end{definition}

We are now ready to construct the map $\varphi_w$ and prove Theorem \ref{thm:redbeer}. One should keep in mind that the injection constructed in Section \ref{sec:freaksult} for the word $w = 1001$ is exactly the map $\varphi_{1001}$: in this case the compact clusters are $\mathfrak{C}^{1001} = \{1001, 1001001\}$, and the marked compact near-clusters are 
\begin{equation}\mathfrak{C}^{1001}(1) = \{1\widehat{1}01, 10\widehat{1}1\} \hspace{.25cm} \text{ and } \hspace{.25cm} \mathfrak{C}^{1001}(0) = \{\widehat{0}001, 100\widehat{0}, 100\widehat{0}001\}. \end{equation}

\begin{proof}[Proof of Theorem \ref{thm:redbeer}] We start by constructing an injective map $h_w: \mathfrak{C}^w(1) \to \mathfrak{C}^w(0)$. By the discussion following Definition \ref{def:center}, and since $2|b|_1 \geq |b|$ for all borders $b$ of $w$, for each $c \in \mathfrak{C}^w$ there exists a matching of the $1$s and $0$s in $I(c)$ such that every $0$ is matched with a unique $1$. Let $M_c$ denote such a matching. For any compact cluster $c \in \mathfrak{C}^w$ and index $j \in I(c)$ with $c_j = 0$, we define $h_w((c, j)^1) = (c, M_c(j))^0$, where $M_c(j)$ denotes the index of the $0$ matched to the $1$ at position $j \in I(c)$. In other words, we move the mark from the originally marked $1$ to its matching $0$, and bitflip the values of both positions. It follows immediately that the map $h_w$ is a well-defined injection. 

We now use the map $h$ to construct $\varphi_w$. Fix $\omega \in \Omega^w(1)$. 

\begin{enumerate}
	
	\item Suppose the marked $1$ in $\omega$ belongs to a marked factor of $\omega$ which is an element of $\mathfrak{C}^w(1)$. In this case, we define $\varphi_w(\omega) = \omega^0 \in \Omega^w(0)$. 

	\item Suppose the marked $1$ does not belong to a marked factor of $\omega$ which is an element of $\mathfrak{C}^w(1)$. Let $c$ denote the (contiguous) subword of $\omega^0$ consisting of all digits which belong to both the leftmost copy of $w$ containing the marked index in $\omega^0$ and to the rightmost such copy of $w$ -- these copies of $w$ exist by the definition of $\mathfrak{C}^w(1)$ (they may be equal). Since all factors $w$ in $\omega^0$ contained in $c$ overlap at the marked position, $c$ is a compact cluster, and since the mark has value $1$ in $\omega$, the corresponding digits of $\omega$ form the factor $c^1 \in \mathfrak{C}^w(1)$. Thus we can define $\varphi_w(\omega)$ by replacing the factor $c^1$ in $\omega$ with $h_w(c^1)$. 

\end{enumerate}

It remains to check that $\varphi_w$ is an injection with image in $\Omega^w(0)$. In case 1, $\varphi_w$ is an involution onto its image, equal to the map which bitflips the value of the mark. Similarly, since $h_w$ is injective, $\varphi_w$ restricted to case 2 can also be thought of as an involution onto its image, where the reverse map is obtained by the same procedure, reversing the roles of $0$ and $1$ (and applying the inverse of the map $h_w$ restricted to its image). Also, note that the domain of case 1 is exactly those marked sequences $\omega \in \Omega^w(1)$ where both choices $\epsilon = 0, 1$ yield a sequence $\omega^{\epsilon} \in \Omega^w(\epsilon)$, and for any such $\omega$, $\varphi_w(\omega)^\epsilon \in \Omega^w(\epsilon)$ for both $\epsilon = 0, 1$. On the other hand, for $\omega$ in the domain of case 2, $\omega^\epsilon \in \Omega^w(\epsilon)$ if and only if $\epsilon = 1$, and $\varphi_w(\omega)^\epsilon \in \Omega^w(\epsilon)$ if and only if $\epsilon = 0$. It follows that the images of cases 1 and 2 under $\varphi_w$ are disjoint, and that the image lies in $\Omega^w(0)$.

If strict inequality $2|b|_1 > |b|$ holds for some $b \in \mathcal{B}(w,w)$, then $h_w$ is not surjective, and $\Omega^w(0) \setminus \text{Im}(\varphi_w)$ is exactly the set of sequences in $\Omega^w(0)$ where the marked compact cluster corresponding to the mark (as defined in case ii) lies outside the image of $h_w$. The number of such sequences is at least a positive fraction of the total by Corollary \ref{cor:parrypals}.  \end{proof}

\begin{remark} \label{rem:guess} Even when the word $w$ has some borders $b$ which are $0$-heavy and others which are $1$-heavy, examining the marked near-compact-cluster sets $\mathfrak{C}^w(\epsilon)$ allows one to make an educated guess about the sign of $\rho_w - \frac{1}{2}$. Take for example the word $w = 100011110001$, which has borders $\mathcal{B}(w,w) = \{1, 10001, w\}$ and compact clusters $\mathfrak{C}^w = \{w, 1000111w, 10001111000w\}.$ Of the twelve marked near-compact clusters with shape $w$, exactly half are in $\mathfrak{C}^w(0)$ and half are in $\mathfrak{C}^w(1)$; for the five with shape $1000111w \in \{0,1\}^{19}$, two are in $\mathfrak{C}^w(0)$ and three are in $\mathfrak{C}^w(1)$; and the remaining shape, which has length $23$, gives a single marked near-compact cluster in $\mathfrak{C}^w(0)$. Thus, we may pair up these marked near-compact clusters to form a (length-preserving) near-bijection between $\mathfrak{C}^w(0)$ and $\mathfrak{C}^w(1)$, leaving one unmatched element of $\mathfrak{C}^w(0)$ of length $23$, and one unmatched element of $\mathfrak{C}^w(1)$ of length $19$. By Fact \ref{fact:parry}, the probability to observe a given word of length $i$ at the marked location has probability roughly $(\lambda^w)^{-i}$, so the length $19$ shape is roughly $(\lambda^w)^{6}$ times more common than the length $23$ shape, up to global constants coming from the right/left eigenvectors. Thus, assuming these constants are small compared to $\lambda^4 \approx 2^4$, we should predict that marked near-clusters marked $1$ are slightly more common, and thus that $\rho^w > \frac{1}{2}$, which is confirmed by an exact calculation. 	
\end{remark}

\subsection{Proof of Theorem \ref{thm:club}}

\begin{proof}
    First we aim to understand letter frequencies in words of fixed length $n$. 
	In a (finite or infinite) sequence $\omega = \omega_1 \omega_2 \cdots$, denote by $N_1(\omega, k)$ the number of ones in $\omega$ up to the $k$th digit. 
	We use and extend the notation from Definition \ref{def:stop!} by writing $Q(\omega, k)=\frac{N_1(\omega, k)}{k}$ for the frequency of $1$s up to the $k$th index, and $Q(\omega, (k_1, k_2])=\frac{N_1(\omega, k_2)-N_1(\omega, k_1)}{k_2-k_1}$ for the frequency over an (integer) interval $1 \leq k_1 < k_2$. Also recall the hitting time $\tau^w$, viewed as a function $\tau^w: \{0, 1\}^* \setminus \Omega^w \to \N$ which records the first appearance of $w$ in $\omega$ when searching from left to right; and recall the letter frequency $\rho^w$ and $\rho^w_n$ from Definition \ref{def:rhos}. 
	
	By a slight abuse of notation we also use $Q(k)$ to denote the random variable equal to the frequency of $1$s up to time $k$ in an i.i.d. Bernoulli$(1/2)$ sequence. We write $\p$ and $\E$ for probability and expectation with respect to the i.i.d. Bernoulli$(1/2)$ measure. 
	
	First note that the frequency of ones over all words of length $n$ is exactly $\frac{1}{2}$:
\begin{equation}\label{eq:letter_dens_base}
    \frac{1}{2} = \frac{1}{2^n}\sum_{\omega \in \{0, 1\}^n}Q(\omega, n)=\frac{1}{2^n}\sum_{\omega \in \Omega_n^w} Q(\omega, n) + \frac{1}{2^n}\sum_{\omega \in \{0, 1\}^n \setminus \Omega_n^w} Q(\omega, n),
    \end{equation}
    The first sum on the right-hand side can be rewritten as
    \begin{equation} \frac{1}{2^n} \sum_{\omega \in \Omega_n^w} \frac{|\omega|_1}{n} = \frac{|\Omega_n^w|}{2^n} \rho^w_n=\mathbb{P}(\tau^w>n)\rho_n^w \end{equation}
    We split the second sum further according to the hitting time $\tau^w$:
    \begin{align} \sum_{\omega \in \{0, 1\}^n \setminus \Omega_n^w} Q(\omega, n) &= \sum_{k=|w|}^{n}\sum_{\substack{\omega \in \{0, 1\}^n \setminus \Omega_n^w \\ \tau^w(\omega)=k}}Q(\omega, n) \\
    	&=\sum_{k=|w|}^{n}\sum_{\substack{\omega \in \{0, 1\}^n \setminus \Omega_n^w \\ \tau^w(\omega)=k}}\left(\frac{k}{n}Q(\omega, k)+\frac{n-k}{n}Q(\omega, (k, n])\right).
    \end{align}
    For fixed $k$ and summing over $\omega$, the first term becomes
    \begin{equation}\mathbb{P}(\tau^w=k) \E(Q(\tau^w )| \tau^w = k) \frac{k}{n}.\end{equation}
    Since $\tau^w$ is a stopping time with respect to the filtration $\mathcal{F}_n = \sigma[\omega_1, \ldots, \omega_n]$ (and measure $\p$), conditionally on the event $\{\tau^w(\omega) = k\}$, the sequence $\omega_{k+1}, \ldots, \omega_n$ is uniform over $\{0, 1\}^{n-k}$. Thus, the second term simplifies as
    \begin{equation}\frac{1}{2^n}\sum_{\substack{\omega \in \{0, 1\}^n \setminus \Omega_n^w \\ \tau^w(\omega)=k}}\frac{n-k}{n}Q(\omega, (k, n])=\frac{1}{2}\mathbb{P}(\tau^w=k)\frac{n-k}{n}.\end{equation}
    Plugging these two formulae into \eqref{eq:letter_dens_base} and re-arranging yields
    \begin{equation} \label{eq:bigsauce}
		\begin{split}
	\frac{1}{2} &=\mathbb{P}(\tau^w>n)\rho_n^w+\sum_{k=|w|}^n\mathbb{P}(\tau^w=k)\mathbb{E}(Q(\tau^w)|\tau^w=k)\frac{k}{n}+\sum_{k=|w|}^n\frac{1}{2}\mathbb{P}(\tau^w=k)\frac{n-k}{n} \\
	&= \mathbb{P}(\tau^w>n)\rho^w_n+\sum_{k=|w|}^n\mathbb{P}(\tau^w=k)\mathbb{E}(Q(\tau^w)|\tau^w=k)\frac{k}{n}\\
	&\hspace{2.7cm}+\frac{1}{2}\left(1-\mathbb{P}(\tau^w>n)-\sum_{k=|w|}^n\frac{1}{2}\mathbb{P}(\tau^w=k)\frac{k}{n}\right)
    	\end{split}
	\end{equation}
       
    The idea is to take an appropriate linear combination using Equation \eqref{eq:bigsauce}.
 	The weights will be given by the expression $\mu(n) = \frac{1}{n+1}$ for $n \in [|w|, N]$ for large $N$. Note that for fixed $k$, as $N \to \infty$ we have
 	\begin{equation}\label{eq:magicmike} \sum_{n=k}^{N}\mu(n)\frac{k}{n}=1+O(N^{-1})\end{equation}
 	since the sum telescopes. Now consider the following rearranged form of $\eqref{eq:bigsauce}$: 
    \begin{equation} 
		\frac{1}{2} \left(\mathbb{P}(\tau^w>n)+\sum_{k=|w|}^{n}\mathbb{P}(\tau^w=k)\frac{k}{n}\right) = \mathbb{P}(\tau^w>n)\rho_n^w + \sum_{k=|w|}^n\mathbb{P}(\tau^w=k)\mathbb{E}(Q(\tau^w)|\tau^w=k)\frac{k}{n}.
	\end{equation}
    Multiply by $\mu(n)$ and sum over $n \in [|w|, N]$ to obtain
    \begin{equation} 
    	\begin{split} 
    		\frac{1}{2}\sum_{n=|w|}^{N}\mu(n) & \left(\mathbb{P}(\tau^w>n)+\sum_{k=|w|}^{n}\mathbb{P}(\tau^w=k)\frac{k}{n}\right) = \\
    	&\sum_{n=|w|}^{N}\mu(n)\left(\mathbb{P}(\tau^w>n)\rho_n^w+\sum_{k=|w|}^n\mathbb{P}(\tau^w=k)\mathbb{E}(Q(\tau^w)|\tau^w=k)\frac{k}{n}\right).
    	\end{split} 
    \end{equation}
    After exchanging (finite) sums and using Equation \ref{eq:magicmike}, we obtain
	\begin{equation}
		\begin{split}
		\sum_{n=|w|}^{N}\mu(n)\mathbb{P}(\tau^w>n)\left(\frac{1}{2}-\rho_n^w\right)
		+ \frac{1}{2}\sum_{k=|w|}^{N}\mathbb{P}(\tau^w=k)
		&= \sum_{k=|w|}^{N}\mathbb{P}(\tau^w=k)\mathbb{E}(Q(\tau^w)\mid \tau^w=k) \\
		&\quad + O(N^{-1}).
		\end{split}
	\end{equation}
    Now take $N \to \infty$ and substitute $\mu(n) = \frac{1}{n+1}$: this becomes
    \begin{equation} \label{eq:frequency_at_hitting_time_vs_frequency_in_forbidden}
    \frac{1}{2}+\sum_{n=|w|}^{\infty}\frac{\mathbb{P}(\tau^w>n)}{n+1}\left(\frac{1}{2}-\rho_n^w\right)= \E[Q(\tau^w)],
    \end{equation}
    as desired. 
\end{proof}

\section{Examples \& Heuristics} \label{sec:guessing}

In this section we collect some examples, counterexamples, and heuristics related to the ordering of words $w$ by their $\rho$ values. 

\subsection{Larger forbidden sets}

So far we have focused our attention on avoiding a single word as a factor. One might naturally wonder about simultaneously forbidding multiple patterns. The same notions of entropy and letter frequency make sense in this context: given a (finite) set of forbidden patterns $\mathcal{F}$, let $\Omega^{\mathcal{F}}$ denote the set of sequences avoiding all words in $\mathcal{F}$ as factors, and $\rho^{\mathcal{F}}$ the limiting frequency of $1$s over sequences in $\Omega^{\mathcal{F}}$. A naive attempt to generalize the monotonicity principle to this case would be to simply count the total number of $1$s over all words in $\mathcal{F}$: one might guess that if $\mathcal{F}$ itself is $0$-heavy, i.e. if 

\begin{equation} 2\sum_{f \in \mathcal{F}} |f|_1 \leq \sum_{f \in \mathcal{F}} |f|, \end{equation}
then $\rho^{\mathcal{F}} \geq 1/2$. This is plainly false, as the following two examples demonstrate. 

\begin{example} \label{ex:domino}
	Let $\mathcal{F}=\{000, 110, 011\}$. Observe that $\mathcal{F}$ itself is $0$-heavy, with $5$ total $0$s and $4$ total $1$s. By transfer matrix analysis,
	$\rho^{\mathcal{F}} \approx 0.4115$. Alternatively, we give a direct combinatorial proof that $\rho^{\mathcal{F}} < 1/2$. Since $|\Omega_n^{\mathcal{F}}|\to \infty$, we can drop the identically one string without changing
	the limiting frequency of $1$s. All other words in $\Omega_n^{\mathcal{F}}$ are of the form $ua_1a_2\dots a_k v$ for some $k$, where 
	$u\in \{\emptyset, 0, 00\}, a_i\in \{10, 100\}, v\in \{\emptyset, 1\}$. Since $u$ and $v$ do not affect the limiting frequency of 1s, 
	it suffices to see that for fixed $u, v$, the number of $100$ blocks grows linearly with $n$ on average. This translates to the elementary
	problem among all compositions of $m=n-|u|-|v|$ into parts of size 2 or 3, the average number of $3$s grows linearly. Denoting the total number of
	such compositions by $T_m$, we have the recurrence $T_m = T_{m-2} + T_{m-3}$, and for the total number $E_m$ of size $3$ parts over all sequences in $T_m$, 
	we have $E_m = E_{m-2} + E_{m-3}+T_{m-3}$. Thus the average number of parts of size $3$ satisfies the recurrence
	$$\frac{E_m}{T_m} = \frac{E_{m-2}+E_{m-3}+T_{m-3}}{T_{m-2}+T_{m-3}}\geq \frac{E_{m-2}+E_{m-3}}{T_{m-2}+T_{m-3}}+ \frac{T_{m-3}}{3T_{m-3}},$$
	where $T_{m-2} \leq 2 T_{m-3}$ since the entropy $\lambda^{\mathcal{F}} \leq 2$. Thus 
	$$\frac{E_m}{T_m}\geq \min\left(\frac{E_{m-2}}{T_{m-2}}, \frac{E_{m-3}}{T_{m-3}}\right)+\frac{1}{3}.$$
	This is sufficient to guarantee linear growth of $\frac{E_m}{T_m}$.
	
	We note that the forbidden list $\mathcal{F}'=\{000, 11\}$ yields essentially the same space as $\mathcal{F}$, and the same limiting frequency $\rho$. This (counter)example is essentially minimal. 
\end{example}

One might object that the forbidden list $\mathcal{F}'$ in Example \ref{ex:domino} does not weight $0$s and $1$s equally: since $000$ is length $3$ while $11$ is length $2$, forbidding $000$ is less probabilistically costly, and thus forbidding it excludes fewer $0$s than $11$ does $1$s. Still, requiring that all words in $\mathcal{F}$ are the same length is not enough to make this naive guess correct, as the next example shows. 

\begin{example} 
	Take $\mathcal{F} = \{0000, 1011\}$: then $\mathcal{F}$ is $0$-heavy, and by direct calculation, $\rho^{\mathcal{F}} \approx .4966.$ 
\end{example}

\subsection{Local patterns \& Covariance}

A natural guess for how local pattern densities are related is the following calculation over an independent sequence of coin flips. Let $Y_i, i \in \N$ be i.i.d. Bernoulli$(1/2)$ random variables, and for any word $w$, let $N_n^w = N_n^w(Y_1, \ldots, Y_n)$ denote the total number of $w$-factors in the sequence $(Y_1, Y_2, \ldots, Y_n)$. For any two words $v,w$, a direct computation gives 
\begin{equation}
	\lim_{n \to \infty} \frac{1}{n} \Cov(N_n^w, N_n^v) = 2^{-|v|-|w|} \left(-(|v|+|w|-1)+ \sum_{b \in \mathcal{B}(v,w) \cup \mathcal{B}(w,v)} 2^{|b|} \right)
\end{equation}
There is a small but important subtlety here: if $v$ is itself a proper factor of $w$ (which was not allowed in our original definition of $\mathcal{B}$), then $\mathcal{B}(v,w) \cup \mathcal{B}(w, v)$ is interpreted as a multiset which includes a copy of $v$ for every factor of $v$ contained in $w$ in addition to any cross-borders of the usual form. For example, $\mathcal{B}(11, 10110110) \cup \mathcal{B}(10110110, 11) = \{1, 11, 11\}.$

For words $v$ and $w$, let $\rho^w(v)$ denote the frequency of $v$ when conditioning on avoiding $w$:  
\begin{equation} \rho^w(v) = \lim_{n \to \infty} \frac{1}{n |\Omega^w_n|} \sum_{\omega \in \Omega_n^w} N^v_n(\omega). \end{equation}
If the random variables $N^v$ and $N^w$ are positively (resp. negatively) correlated under i.i.d. measure $\p$, we might naively guess that $\rho^w(v) < 2^{-|v|}$ (resp. $\rho^w(v) > 2^{-|v|}$), the ambient frequency of $v$s in $\p$. This makes some heuristic sense: forbidding $w$ is like conditioning on $N^w$ being small, so if $N^w$ and $N^v$ are positively (resp. negatively) correlated, then $N^v$ should also be small (resp. large) on this event. Unfortunately, this is false in general: for example, the words $v = 0011, w = 0000$ are negatively correlated but $\rho^w(v) \approx .06231 < 1/16$ by a transfer matrix computation. (Interestingly, in this case, $\rho^v(w) \approx .06438 > 1/16$.)

Taking $v = 1$, which corresponds to our letter frequency $\rho^w = \rho^w(1)$, we obtain the special case 

\begin{equation} \lim_{n \to \infty} \frac{1}{n} \Cov(N_n^w, N_n^1) = 2^{-|w|-1} (2|w|_1-|w|). \end{equation}

Note that this expression is monotone increasing in $|w|_1$, in agreement with the monotonicity principle discussed in Section \ref{subsec:conj}. However, the covariance is a rather coarse statistic compared to $\rho$ and $q$: the former is constant among words of the same length and same number of $1$s, while as we have seen, $\rho$ and $q$ are not. 

\subsection{Extremal $\rho$ values} 

Among all words of a fixed length $k$, forbidding the word $10^{k-1}$ or its reversal should maximize the frequency $\rho$, although we do not have a proof of this. The optimizer is less clear when comparing words with both fixed length $k$ and a fixed total number of $1$s $j$. Computer computations suggest that, for any fixed $j, k$ with $j < k/2$ and $j \neq \{0, 1\}$, the maximum $\rho$ value is achieved for the word $w$ which has the minimum possible entropy $\lambda^w$, while the minimum is achieved by any word with the maximum possible entropy. Note that there always exists a word with $\mathcal{B}(w,w) = \{w\}$, namely $1^j 0^{k-j}$, which achieves the minimum possible entropy. The maximum entropy is, by Theorem \ref{thm:juicy}, achieved (roughly speaking) by the word with the most, largest borders. For example, an explicit computer check confirms that any word $w$ with $|w| = 10$, $|w|_1 = 3$ and $\mathcal{B}(w,w) = \{w\}$ achieves the maximum $\rho$ value over all such words, while the minimum $\rho$ value over the same set is achieved for the word $0(100)^3$, which has many (large) borders because of its periodic structure. It is not clear whether there is a simple way to describe the total ordering by $\rho$ values within such a class. Explicit computer computations suggest that for $j < k/2$, larger entropy typically gives a smaller $1$s frequency, but this is not always true: an example is the pair $v = 110^511$, $w = 010110010$ (with $k = 9, j = 4$), where $\lambda^v < \lambda^w$, but $\rho^v < \rho^w$.

One may object that among the words with the same length and same number of $1$s, since different entropy values $\lambda$ can occur, the resulting shift spaces and their measures of maximal entropy are fundamentally incomparable. Here is one possible remedy: consider the ordering of $\rho$ values among words with a fixed border set $\mathcal{B}(w,w)$, which by Theorem \ref{thm:juicy} fixes the entropy. Among these, computer computations suggest that words with higher powers of $y$ in their border polynomials give smaller values of $\rho$: for example, by direct computation, $\rho^{100001} < \rho^{000110}$, with corresponding border polynomials $x^6y^2 + xy$ and $x^6y^2 + x$, respectively. (Note also that although the assumptions of Theorem \ref{thm:coffee} hold for the pair $100001$ and $111001$, the latter word being the bitflip of $000110$, the conclusion of Theorem \ref{thm:coffee} does not imply the stated comparison between the $\rho$ values as a corollary: it only gives that $\rho^{100001} > \rho^{111001} = 1-\rho^{000110}$.) This agrees with the idea of the proof of Theorem \ref{thm:redbeer}, but we are not aware of a general class of injective maps which demonstrate these types of inequalities other than the $\psi_{v,w}$ maps arising in Theorem \ref{thm:coffee}. 

\section{Future directions} \label{sec:further}

Finally, we collect some future directions left open by our results. 

\begin{enumerate}[label={}, leftmargin=1.5em]
	
\item ({\bf{Arbitrary SFTs $\&$ Local Patterns}}) Our notions of pattern frequency make sense for shifts of finite type over any finite alphabet (rather than just the binary alphabet), with base space being any Cayley graph (instead of our situation in $\mathbb{Z}$), with any forbidden set (not just single words), and any arbitrary unions/intersections of cylinder sets (not just the cylinder set of $1$). For a larger base alphabet, what is the right generalization of the balanced borders condition for perfect letter symmetry? The situation for SFTs over $\mathbb{Z}^2$ is already significantly more difficult from a dynamical systems perspective: for example, the ordering of shift spaces by their entropy is unknown. For more general forbidden patterns and local statistics, what new phenomena arise? Our methods can likely be applied to give some results in these situations.

\item ({\bf{Ordering by frequency}}) Our theorems and the computer computations we performed do not give any conclusive answer, or even a natural guess, for how the $\rho$ and $q$ values are ordered among words with the same length and the same number of $1$s, or among words with the same length and the same entropy. Can these orderings be described in a combinatorially simple way? Part of the difficulty in answering these and similar questions is that the space of polynomials that arise as border polynomials is highly structured, but in a way that is combinatorially complex. One approach may be to generalize the characterization of (one-variable) border-polynomials described in~\cite[Theorem 5.1]{GuibasOdlyzko1981Periods} to the two-variable case. Alternatively, one could try to generalize the proof in~\cite[Section 7]{GuibasOdlyzko1981StringOverlaps}, which compares recursions directly, to the two-variable case. This is one possible route to removing the assumption $\rho_n^v \leq \rho_n^w$ from Corollary \ref{cor:kave}. 

\item ({\bf{Various orderings}}) As we observed previously, the $\rho$ and $q$ orderings are not perfect inverses of each other. Similarly, for more general SFTs, the limiting covariance statistic described in Section \ref{sec:guessing} often predicts how the pattern frequency will change compared to the full shift (where no words are forbidden). Can one describe when and why these differences occur, precisely or heuristically? 

\item ({\bf{Entropy differences}}) All our theorems only give explicit injections between shift spaces with equal entropy. Is it possible to give nice maps between $\Omega^v, \Omega^w$ in the case where $\lambda^v \neq \lambda^w$? Although embeddings/factor maps between the corresponding shift spaces exist under an easy-to-check periodic-point condition~\cite[Theorems 10.1.1, 10.3.1]{LindMarcus2021SymbolicDynamics} and are necessarily sliding block codes (i.e. they are local maps), it is not clear how to construct an explicit such map, even for special cases. Moreover, even if such a map exists, it may not pushforward the Parry measure on the domain to the Parry measure on the target, and it may not be useful for comparing letter densities. 

\item ({\bf{Frequency matches and Irreducibility}}) Note that a priori, $\rho^w = q^w = 1/2$ is possible without having exact $0 \leftrightarrow 1$ symmetry, and Proposition \ref{prop:baboosh} does not rule out this possibility. However, a computer search suggests that this symmetry is necessary for $\rho = q = 1/2$: 

\begin{conjecture} \label{conj:crispy} A word $w$ has $\rho^w = 1/2$ if and only if $w$ has balanced borders. The same holds with $\rho^w$ replaced by $q^w$. \end{conjecture}

In other words, it seems to be impossible for different borders which are more or less $1$-favored to cancel each other out to yield frequency exactly $1/2$. This is reminiscent of~\cite{GuibasOdlyzko1981StringOverlaps} (see Theorem 1.8 and the discussion following it), where the authors speculate about irreducibility of certain border polynomials. 

\end{enumerate}

\section*{Acknowledgements}
\addcontentsline{toc}{section}{Acknowledgements}

Jacob Richey was supported by a Simons Foundation grant awarded to the Rényi Institute. 
Mikl\'os B\'ona was supported by Simons Foundation Award 940024. 
Balázs Maga was supported by the NKFIH Grants no. 152535 and no. 152922, funded by the Ministry of Innovation and Technology of Hungary from the National Research, Development and Innovation Fund. We thank the three readers who pointed out a size four counterexample to the monotonicity conjecture in the original version of this paper, and thus catalyzed this revision. 


\appendix
\titleformat{\section}{\normalfont\Large\bfseries}{Appendix \thesection:}{1em}{}

\section{Convergence to Parry Chain} \label{app:parryconv}
	
In this appendix we recall the construction and some properties of a finite state Markov chain corresponding to a given shift of finite type, commonly known as the Parry chain (named for William Parry \cite{Parry1964IntrinsicMarkovChains}), and prove convergence of the uniform measure on the set $\Omega_n^w$ to a sample from the Parry chain. The basis of this connection is that for any shift of finite type (including our spaces $\Omega^w$ as a special case), the set of allowable sequences can be realized as the set of paths in a finite graph. There may be many possible choices for such a graph, sometimes called an {\bf{edge-shift graph}}, and the corresponding adjacency matrix is commonly called a {\bf{transfer matrix}} for the shift space. For the shift spaces $\Omega^w$ considered in our work, the minimal valid edge-shift graph is called the {\bf{follower set graph}}, see \cite[p.73]{LindMarcus2021SymbolicDynamics} for the general version or \cite[Sec 3.1]{ChandgotiaMarcusRicheyWu2026SinglePattern} for the one-word SFT case. Another valid choice is the graph obtained from the de Bruijn graph on binary sequences of length $k-1$ by deleting a single edge corresponding to the transition $w_1\cdots w_{k-1} \to w_2\cdots w_k$ where $|w| = k$. 

The Parry chain (and corresponding Parry measure) makes sense and has nice properties for any primitive matrix: 

\begin{definition} Fix a primitive matrix $A$ with finite state space $S$ and $\{0, 1\}$ entries, and let $\mu, r$ denote the Perron-Frobenius eigenvalue and corresponding eigenvector of $A$, with any normalization for $r$ such that $r$ is strictly positive. Denote by $P = P(A)$ the {\bf{Parry chain of A}}, i.e. the matrix with entries 
	\begin{equation} P_{xy} = A_{xy} \frac{r_y}{\mu r_x}, \quad x, y \in S. \end{equation}
\end{definition}

\begin{fact} \label{fact:parry} The Parry chain $P$ satisfies the following properties:
	\begin{enumerate} \item $P$ has stationary measure, called the {\bf{Parry measure}}, $\pi$ given by $\pi_x = \frac{\ell_x r_x}{\sum_y \ell_y r_y}$, where $\ell$ is the left eigenvector of $A$ with eigenvalue $\mu$, and $\pi_x > 0$ for all $x \in S$.
		\item $P$ is a stochastic matrix
		\item \label{eq:parryform} For $n \geq 1$ and all $x, y \in S$,  
		\begin{equation} (P^n \pi)_{xy} = \mu^{-n} \ell_x r_y \end{equation}
	\end{enumerate}
\end{fact}

\begin{theorem} \label{thm:parrymain} Let $A$ be any primitive transfer matrix with Parry chain $P$. Denote by $S_n$ the set of allowed paths in $A$ of length $n \geq 1$:
	\begin{equation} \label{eq:allowed} S_n = \{s_1s_2\cdots s_n: A_{s_i s_{i+1}} = 1 \text{ for } i = 1, \ldots, n-1\}.\end{equation}
	Let $\mathbb{P}_{n}$ be the measure induced by $P$ on $S_n$, as in \ref{eq:parryform} of Fact \ref{fact:parry}, and for integers $i, j$, let $\mathbb{Q}_{[i, n-j]}$ be the uniform measure on $S_n$ restricted to the symbols $(s_i, \ldots, s_{n-j})$. Fix any integer sequences $i(n),j(n)$ such that $i(n), j(n) \ll n$ and $i(n), j(n) \to \infty$ as $n \to \infty$. Then 
	\begin{equation} \TV(\p_{n-i(n)-j(n)}, \mathbb{Q}_{[i(n), n-j(n)]}) \to 0 \text{ as }  n \to \infty \end{equation}
	where $\TV$ denotes total variation distance. 
\end{theorem}

\begin{proof}
	Let $\mu, \ell, r$ denote the Perron-Frobenius eigenvalue and corresponding left/right eigenvectors of $A$, normalized so that
	\begin{equation}
		\sum_x \ell_x r_x = 1.
	\end{equation}
	Since $A$ is primitive, the Perron-Frobenius theorem implies that there exists $\eta < \mu$ such that
	\begin{equation}\label{eq:entrywisePF}
	(A^m)_{xy}
	=
	\mu^m r_x \ell_y
	\left(1 + O\left((\eta/\mu)^m\right)\right),
	\end{equation}
	where the error is uniform in $x,y \in S$. Write
	\begin{equation}
		L_n = n - i(n) - j(n),
	\end{equation}
	and fix an admissible word
	\begin{equation}
		t = t_1 \cdots t_{L_n} \in S_{L_n}.
	\end{equation}
	We compute
	\begin{equation}\label{eq:block-probability}
		\mathbb{Q}_{[i(n),n-j(n)]}(w)
		=
		\frac{
			\left(\sum_x (A^{i(n)-1})_{x t_1}\right)
			\left(\sum_y (A^{j(n)-1})_{t_{L_n} y}\right)
		}{
			\sum_{x,y}(A^{n-1})_{xy}
		}.
	\end{equation}
	Applying \eqref{eq:entrywisePF} to the numerator and denominator yields
	\begin{equation}\label{eq:RN}
		\mathbb{Q}_{[i(n),n-j(n)]}(t)
		=
		\mu^{-(L_n-1)}
		\ell_{t_1} r_{t_{L_n}}
		(1 + \varepsilon(t_1, t_{L_n})),
	\end{equation}
	where the error $\varepsilon(t_1, t_{L_n}) = O((\eta/\mu)^{i(n) \vee j(n)}$ uniformly in $t_1, t_{L_n}$. On the other hand, by Fact \ref{fact:parry},
	\begin{equation}
		\mathbb{P}_{L_n}(w)
		=
		\mu^{-(L_n-1)}
		\ell_{t_1} r_{t_{L_n}}.
	\end{equation}
	Therefore
	\begin{equation}
		\frac{
			\mathbb{Q}_{[i(n),n-j(n)]}(t)
		}{
			\mathbb{P}_{L_n}(t)
		}
		=
		1 + \varepsilon_n(t_1,t_{L_n}).
	\end{equation}
	Finally,
	\begin{align}
		\TV(\mathbb{P}_{L_n},\mathbb{Q}_{[i(n),n-j(n)]})
		&=
		\frac12
		\sum_{t \in S_{L_n}}
		\left|
		\mathbb{P}_{L_n}(t)
		-
		\mathbb{Q}_{[i(n),n-j(n)]}(t)
		\right|
		\\
		&=
		\frac12
		\sum_{t \in S_{L_n}}
		\mathbb{P}_{L_n}(t)
		|\varepsilon(t_1,t_{L_n})|
		\\
		& \leq \frac{1}{2} \sup_{s, s' \in S^2} |\varepsilon(s, s')| \to 0. 
	\end{align}
\end{proof}

\begin{corollary}\label{cor:parrypals} Fix a binary word $w$ and any word $v$ which is not a factor of $w$. Let $N_n^v$ denote the random variable which counts, for a sequence $\omega \in \Omega^w$, the number of factors $v$ in $\omega$, i.e. the number of $i \in [n]$  such that $\omega_{i}\omega_{i+1}\cdots\omega_{i+|v|-1} = v$. Then there exists a constant $c = c(v,w) > 0$ such that for any $\epsilon > 0$,
	\begin{equation} \mathbb{Q}_n(N_n^v < (c-\epsilon) n) \to 0 \quad \text{ as } n \to \infty. \end{equation}
\end{corollary}

\begin{proof} Assume $w \notin \mathcal{R}$ (c.f. Section \ref{subsec:conj}): we omit the cases $w \in \mathcal{R}$, which can be dealt with by explicit calculations. By~\cite[Section 3]{ChandgotiaMarcusRicheyWu2026SinglePattern}, there exists an edge-shift (di)graph $G = G^w$ and corresponding transfer matrix $A = A^w$ which is primitive. Moreover, denoting by $S$ the vertex set of $G$, this representation gives a bijection between $\Omega^w$ and the set $S_n$ (c.f. Equation \eqref{eq:allowed}) with the property that each $v$ factor in $\omega \in \Omega^w$ corresponds to a factor $\tilde{v} \in S^{|v|}$ at the same position, and vice versa; and moreover, since $v$ is not a factor of $w$, $\tilde{v}$ does occur as an allowable vertex path in $G$. By Theorem \ref{thm:parrymain} and the LLN for Markov Chains applied to the Parry chain $P$ \cite[Theorem C.1]{LevinPeresWilmer2017MarkovChains}, 
	\begin{equation} \mathbb{Q}_n\left(\left|n^{-1} N_n^v - \mu^{-|v|} \pi(\tilde{v}_1)\pi(\tilde{v}_{|v|}) \right| > \epsilon\right) \to 0 \quad \text{ as } n \to \infty, \end{equation}
	where $\mu = \lambda^w$ and $\pi$ is the Parry measure. By Fact \ref{fact:parry}, this limit is strictly positive. 
\end{proof}
\end{document}